\documentclass[11pt,a4paper,reqno]{amsart}
\usepackage{amsmath,amssymb,amsfonts,epsfig,mathrsfs,cite, hyperref}
\usepackage[T1]{fontenc}
\usepackage{color}
\usepackage{array}
\usepackage{amsthm}
\usepackage{amstext}
\usepackage{graphicx}
\usepackage{setspace}

\usepackage[title]{appendix}

\makeatletter
\@namedef{subjclassname@2020}{%
  \textup{2020} Mathematics Subject Classification}
\makeatother

\usepackage[margin=2.5cm]{geometry}
\usepackage{color}
\usepackage{enumitem}
\usepackage{amscd,psfrag}
\usepackage{yhmath}
\usepackage[mathscr]{eucal}

\usepackage{comment}

\allowdisplaybreaks[4]

\usepackage{slashed}

\makeatletter
\pdfpageheight\paperheight
\pdfpagewidth\paperwidth

\usepackage{epstopdf}

\usepackage{indentfirst}	

\usepackage[normalem]{ulem}
\theoremstyle{plain}

\newtheorem{definition}{Definition}
\newtheorem{theorem}[definition]{Theorem}
\newtheorem*{theorem*}{Theorem}

\newtheorem{remark}[definition]{Remark}

\newtheorem*{remark*}{Remark}
\newtheorem*{sideremark*}{Side Remark}

\newtheorem*{claim*}{Claim}
\newtheorem*{q*}{Question}
\newtheorem{lemma}[definition]{Lemma}

\newtheorem*{corollary*}{Corollary}

\newtheorem{proposition}[definition]{Proposition}

\newcommand{\R}{\mathbb{R}}

\newcommand{\na}{\nabla}
\newcommand{\emb}{\hookrightarrow}

\newcommand{\p}{\partial}

\newcommand{\e}{\varepsilon}

\newcommand{\dd}{{\rm d}}

\newcommand{\G}{\Gamma}
\newcommand{\linf}{{L^\infty}}

\newcommand{\tor}{{\mathbf{T}^2}}
\newcommand{\brator}{{\left(\tor\right)}}
\newcommand{\etwo}{{\mathbf{e}_2}}
\newcommand{\bes}{\mathcal{B}}

\newcommand{\bigt}{{\mathscr{B}_t}}
\newcommand{\smallt}{{\mathscr{S}_t}}

\newcommand{\tnu}{{\widetilde{\nu}}}

\newcommand{\ttnu}{{\widetilde{\tnu}}}

\def\XXint#1#2#3{{\setbox0=\hbox{$#1{#2#3}{\int}$ }
\vcenter{\hbox{$#2#3$ }}\kern-.6\wd0}}

\newcommand{\area}{{\rm Area}}

\title{Inviscid Limit for Yudovich solution to heat conductive Boussinesq equation on two-dimensional periodic domain}

\author{Siran Li}

\address{Siran Li: School of Mathematical Sciences $\&$ CMA-Shanghai, Shanghai Jiao Tong University, No.~6 Science Buildings,
800 Dongchuan Road, Minhang District, Shanghai, China (200240)}

\email{\texttt{siran.li@sjtu.edu.cn}}


\date{\today}

\begin{document}

\begin{abstract}

We establish the inviscid limit of the Yudovich solution to the heat conductive Boussinesq equation with initial velocity and temperature/buoyancy in $L^2$ and initial vorticity in $\linf$ on the two-dimensional periodic domain $\tor$. Given any finite time $T>0$ and $p \in [1,\infty[$, we show that the solution to the diffusive Boussinesq equation converges in $\linf\left(0,T; W^{1,p}\brator\right)$ to the solution to the Euler--Boussinesq equation as the viscosity tends to zero, provided that the initial vorticity, velocity, and temperature/buoyancy converge strongly in $L^2$. Our proof adapts and extends the arguments in [P. Constantin, T.~D. Drivas, and T.~M. Elgindi, \emph{Comm. Pure Appl. Math.} \textbf{75} (2022), 60--82] to forcing terms in  $L^1\left(0,T; \linf\brator\right)$.

\end{abstract}
\maketitle

\section{Introduction}\label{sec: intro}

We are concerned with the inviscid limit of the viscous, heat-conductive (or buoyancy-driven) Boussinesq equations on the two-dimensional periodic domain $\tor$: 
\begin{equation}\label{eq: bkn}
    \begin{cases}
\p_t u^\nu + u^\nu\cdot\na u^\nu - \nu \Delta u^\nu + \na P^\nu = \theta^\nu \etwo,\\
\p_t \theta^\nu + u^\nu \cdot \na \theta^\nu - \kappa\Delta\theta^\nu = 0,\\
\na \cdot u^\nu = 0 \qquad \text{in } [0,T] \times \tor,\\
(u^\nu, \theta^\nu) = (u^\nu_0, \theta^\nu_0)\qquad \text{at } \{0\}\times \tor.
    \end{cases}
\end{equation}
By setting $\nu=0$ in the Boussinesq equation~\eqref{eq: bkn} above, we formally obtain the following system, referred to as the Euler--Boussinesq equation in this work: 
\begin{equation}\label{eq: bk0}
    \begin{cases}
\p_t u + u\cdot\na u + \na P = \theta \etwo, \\    
\p_t \theta + u\cdot \na \theta - \kappa\Delta\theta = 0,\\
\na \cdot u = 0 \qquad \text{in } [0,T] \times \tor,\\
(u, \theta) = (u_0, \theta_0)\qquad \text{at } \{0\}\times \tor.
    \end{cases}
\end{equation}
We shall prove the strong convergence $u^\nu \to u$ in $\bigcap_{1\leq p < \infty}\linf\left(0,T; W^{1,p}\brator\right)$ as $\nu \searrow 0$, where $(u,\theta)$ is the unique Yudovich weak solution (see Theorem~\ref{thm: DP} below) to the Euler--Boussinesq equation~\eqref{eq: bk0}, provided that $(u^\nu_0, \theta_0^\nu) \to (u_0, \theta_0)$ and $\omega^\nu_0 := {\rm rot}(u^\nu_0) \to \omega_0:={\rm rot}(u_0)$ in $L^2\brator$.

The Boussinesq equations~\eqref{eq: bkn}, \eqref{eq: bk0} are of physical significance in the modelling of atmospheric and/or oceanographic turbulence; \emph{cf.} Pedlosky \cite{ped}. The variables $u^\nu$, $u$ are the velocity fields, $P^\nu$, $P$ are the pressure fields, and $\theta^\nu \etwo$, $\theta \etwo$ are the vertical forces (where $\etwo = [0,1]^\top$) generated by heat conduction or buoyancy in fluid flows. The positive parameters $\nu$ and $\kappa$ in \eqref{eq: bkn} represent the coefficients of viscosity and heat conductivity.

The qualitative behaviours of large-amplitude solutions to the two-dimensional Boussinesq equations have been a central, extensively studied topic in the literature of mathematical fluid dynamics. We refer the reader to
\begin{itemize}
    \item 
\cite{5,17,18,19,20,22,27,37,51,53,x} for local well-posedness, breakdown criteria, explicit solutions, and singularity formation for the degenerate case ($\nu=\kappa=0$);
\item 
\cite{1,2,3,4,14,15,21,32,33,34,35,36,38,39,44,45,46,48,chae, DP1, DP2, zhao, 49, a1} for global well-posedness and regularity for partially degenerate cases ($\nu>0, \kappa=0$ or $\nu=0, \kappa>0$);
\item 
\cite{41,42,47,52,59,60,61,b1} for well-posedness and regularity with (super-)critical dissipation; and
\item
\cite{10,dwzz,44,54,58,zhao,c1,c2,c3,c4,c5,c6,c7,c8,c9,c10,c11,c12} for large-time behaviours.
\end{itemize}
The above list of literature is by no means exhaustive.

The inviscid limit of the Boussinesq equation~\eqref{eq: bkn} is an important mathematical problem. Moffatt pointed out the significance of studying whether having $\nu>0$ or $\kappa>0$ precludes the finite-time formation of singularities \cite{mof}. Danchin--Paicu~\cite{DP0} proved that for $\nu>0$ and $\kappa=0$, no singularity is formed for finite-energy initial data. The case $\nu=0$ and $\kappa>0$ (\emph{i.e.}, the Euler--Boussinesq equation~\eqref{eq: bk0}), however, encompasses considerably more difficulty. Chae \cite{chae} established the global well-posedness of~\eqref{eq: bk0} in $H^s$ for $s \geq 3$ by energy considerations together with the endpoint Sobolev embedding inequality \emph{\`{a} la} Brezis--Wainger, and it has been extended to rough data in Hmidi--Keraani \cite{33} for $(u_0, \theta_0) \in \bes^{1+2/p}_{p,1} \times L^r$ with some $2 <r \leq p \leq \infty$.\footnote{An extra technical condition is needed for $p=r=\infty$.}

In line with~\cite{chae, 33}, by exploiting delicate Littlewood--Paley estimates, Danchin--Paicu (2009) established the following \cite[Theorem~1]{DP}:\footnote{Indeed, the theorem is proved for the fluid domain being the whole plane $\R^2$, with initial data $\theta_0 \in L^2 \cap \bes^{-1}_{\infty,1}$, $u_0 \in L^2$, $\omega_0\in L^r\cap L^\infty$ for some $r \in [2,\infty[$ and solution $\theta \in C^0_t \left(L^2 \in \bes^{-1}_{\infty,1}\right)_x \cap L^2_t H^1_x \cap L^1_t \left(\bes^1_{\infty,1}\right)_x$, $u \in C^{0,1}_t L^2_x$, and $\omega \in L^\infty_t (L^r\cap\linf)_x$. As commented on \cite[p.1]{DP}, ``our results extend with no difficulty to periodic boundary conditions, though''. }
\begin{theorem}\label{thm: DP}
Fix any $T>0$. Consider the Euler--Boussinesq Equation~\eqref{eq: bk0} in $[0,T]\times\tor$ with initial data $(u_0,\theta_0) \in L^2 \times L^2$ and $\omega_0 \in L^\infty$. There exists a unique solution $(u,\theta)$ such that
\begin{align*}
    u \in C^{0,1}_t L^2_x, \quad \omega \in L^\infty_t L^\infty_x,\quad\text{and }  \theta \in C^0_t L^2_x \cap L^2_t H^1_x \cap L^1_t \left(\bes^1_{\infty,1}\right)_x.
\end{align*}
\end{theorem}
We refer to $(u,\theta)$ in Theorem~\ref{thm: DP} as the \emph{Yudovich solution} to the Euler--Boussinesq equation~\eqref{eq: bk0}, named after V.~I. Yudovich who pioneered the study \cite{yud} of uniqueness of weak solutions to the incompressible Euler equations with bounded vorticity in 1963.

An adaptation of the arguments in \cite{DP} leads to the existence and uniqueness of the Yudovich solution $(u^\nu, \theta^\nu)$ to the diffusive Boussinesq Equation~\eqref{eq: bkn}:
\begin{align*}
    u^\nu \in C^{0,1}_t L^2_x, \quad \omega^\nu \in L^\infty_t L^\infty_x,\quad\text{and }  \theta^\nu \in C^0_t L^2_x \cap L^2_t H^1_x \cap L^1_t \left(\bes^1_{\infty,1}\right)_x
\end{align*}
for any initial data $(u^\nu_0,\theta^\nu_0) \in L^2 \times L^2$ and $\omega^\nu_0 \in L^\infty$.  
In fact, standard energy estimates yield that when $\nu, \kappa>0$, there exists a unique global solution for arbitrarily large data  \cite{14, guo}.

The main theorem of this paper is as follows. It suggests that no turbulence onsets for the 2D periodic Boussinesq flow in the inviscid limit $\nu \searrow 0$ (with fixed $\kappa>0$). 
\begin{theorem}\label{thm: main}
    Let $(u,\theta)$ be the unique Yudovich weak solution to the heat conductive Euler--Boussinesq Equation~\eqref{eq: bk0} with initial data $(u_0, \theta_0) \in L^2\left(\tor;\R^2\right) \times L^2\brator$ and $\omega_0 \in L^\infty\brator$. Let $(u^\nu, \theta^\nu)$ be the unique Yudovich weak solution to the heat conductive, diffusive Boussinesq Equation~\eqref{eq: bkn} with initial data $(u^\nu_0, \theta^\nu_0) \in L^2\left(\tor;\R^2\right) \times L^2\brator$ and $\omega^\nu_0 \in L^\infty\brator$.

    Suppose that $\omega_0^\nu \to \omega_0$, $u^\nu_0\to u_0$, and $\theta^\nu_0 \to \theta_0$, all strongly in $L^2\brator$.  Then, for any finite time $T>0$ and $p \in [1,\infty[$, the inviscid limit for vorticity holds strongly in $L^\infty\left(0,T; L^p\brator\right)$:
    \begin{align*}
        \lim_{\nu \searrow 0} \sup_{t \in [0,T]} \left\| \omega^\nu(t) - \omega(t)\right\|_{L^p\brator} = 0.
    \end{align*}
\end{theorem}

\begin{remark}
By Poincar\'{e}'s inequality, the strong convergence of $u_0^\nu \to u_0$ in $L^2\brator$ can be deduced from $\omega_0^\nu \to \omega_0$ in $L^2\brator$, \emph{provided that} $\int_\tor u^\nu_0\,\dd x \to \int_\tor u_0\,\dd x$.

The issues pointed out in \cite[Remark~1]{cde} carry over to the case of Boussinesq equations on $\tor$. More precisely, (i) Theorem~\ref{thm: main} does not hold for $T=+\infty$; (ii) there cannot be any effective estimate on the convergence rate without additional regularity assumptions on $\omega_0$; (iii), the convergence $\omega^\nu\to\omega$ does not hold in $L^\infty_t L^\infty_x$; and (iv), the arguments in this paper only work for $\tor$, while on general bounded domains one expects the development of boundary layers.

\end{remark}

The proof of Theorem~\ref{thm: main} relies on the analysis of the vorticity equations:
\begin{align}
    & \p_t \omega^\nu + u^\nu \cdot \na \omega^\nu - \nu \Delta \omega^\nu = \p_1 \theta^\nu,\label{vort eq, nu} \\
    & \p_t \omega + u \cdot \na \omega = \p_1\theta.\label{vort eq, 0} 
\end{align}
obtained by taking ${rot}$ to the Boussinesq and Euler--Boussinesq equations~\eqref{eq: bkn} and \eqref{eq: bk0}. They can be viewed as the vorticity equations for the Navier--Stokes and Euler equations with forcing terms $\p_1\theta^\nu$ and $\p_1\theta$, respectively. In the seminal work \cite{cde} by Constantin--Drivas--Elgindi (2022), the inviscid limit of the Yudovich solution to the Navier--Stokes equations on $\tor$ \emph{with $L^\infty_tL^\infty_x$-forcing} has been established. In this paper, we show that the approach in \cite{cde} is, in fact, amenable to be extended to $L^1_tL^\infty_x$-forcing, which is natural for the Yudovich solution to the Euler--Boussinesq equation~\eqref{eq: bk0} as in the framework \emph{\`{a} la} Danchin--Paicu \cite{DP}. Such extensions and improvements are achieved by essentially ODE-type arguments (see \S\ref{sec: estimates} below). This is the key ingredient of our proof for the Main Theorem~\ref{thm: main}.

Let us further remark on the significance and the ``naturality'' of the regularity condition $\theta^\nu, \theta \in L^1_t W^{1,\infty}_x$. Indeed, for strong solution $(u, \theta) \in C^0_t (H^s_x \times H^s_x)$ for $s \geq 3$ to the Euler--Boussinesq equation~\eqref{eq: bk0}, the condition $$\int_0^{t_\star} \|\na\theta(t)\|_{L^\infty_x}\,\dd t = \infty$$ is a \emph{breakdown criterion}. That is, it implies that the strong solution must blow up (in $H^3_x$-norm) before time $t_\star$; see \cite{18, 19, 53}. This is the analogue for the well-known Beale--Kato--Majda breakdown criterion for the Euler equations.

\smallskip
\noindent
{\bf Notation.} Throughout this paper, we fix a finite time $T>0$. The spatial domain is the two-dimensional periodic domain $\tor = [0,1]^2 = \R^2/\mathbf{Z}^2$, with opposite sides of the square identified with each other. For $\Omega \subset \tor$, $\area(\Omega)$ denotes the 2-dimensional Lebesgue measure of $\Omega$.

We write the Bochner spaces $L^p_tL^q_x \equiv L^p\left(0,T; L^q\brator\right)$, and similarly for $C^0_tL^q_x$, $L^p_t \left(\bes^1_{\infty,1}\right)_x\ldots$ We denote by $\bes^s_{p,q}$ the Besov spaces; in particular, $\bes^1_{\infty,1}\brator \emb W^{1,\infty}\brator$ continuously.

\smallskip
\noindent
{\bf Organisation.}
The remaining parts of the paper are organised as follows:

In \S\ref{sec: prelim} we collect several technical inequalities that are used repeatedly in the paper. 

In \S\ref{sec: estimates} we prove two key estimates, which are extensions of the results in \cite{cde} to the case of $L^1_tL^\infty_x$-forcing terms. 

In \S\ref{sec: conclusion} we present the proof of our Main Theorem~\ref{thm: main}.

We close the paper by several concluding remarks in \S\ref{sec: comments}.

\section{Preliminaries}\label{sec: prelim}

In this section, we collect a few lemmata needed for subsequent developments.

First, we have a simple arithmetic lemma, which serves as a substitute for the usual Young's inequality $ab \leq \frac{a^p}{p}+\frac{b^{p'}}{p'}$ in the case $a<0$.

\begin{lemma}\label{lem: arithmetic}
For any $a \in \R$ and $b>0$, we have 
\begin{align*}
    ab \leq e^a + b \log b - b.
\end{align*}
\end{lemma}

Next, recall the Gr\"{o}nwall inequality of the integral form.
\begin{lemma}\label{lem: gronwall}
    For non-negative functions $y(t)$, $A(t)$, and $B(t)$ on $[0,T]$ with $A,B \in  L^1_t([0,T])$, if $y'(t) \leq A(t)y(t) + B(t)$ for all $t \in [0,T[$, then for any $0 \leq t_0 \leq t$ we have that
    \begin{align*}
        y(t) \leq y(t_0) \cdot e^{\int_{t_0}^t A(s)\,\dd s} + \int_{t_0}^t B(s) \left\{e^{\int_{s}^t A(\tau)\,\dd \tau}\right\}\,\dd s.
    \end{align*}
\end{lemma}

We shall also use the following Trudinger--Moser inequality; see \cite[Lemma~1]{cde} (one may choose $C_K=2\area\brator=2$). It quantifies the $\linf\to {\rm BMO}$ boundedness of the Calder\'{o}n--Zygmund operator $\na \na^\perp(-\Delta)^{-1}$, \emph{i.e.}, gradient of the convolution by the Biot--Savart kernel. 

\begin{lemma}\label{lem: MT}
Let $\omega = {\rm rot}\, u \in \linf\brator$. Then there exist universal constants $\gamma$, $C_K>0$ such that for any $\beta>0$ with $\beta \|\omega\|_{L^\infty_x} \leq \gamma$, it holds that
\begin{align*}
    \int_{\tor} e^{\beta|\na u|}\,\dd x \leq C_K.
\end{align*}
\end{lemma}

Here and hereafter, we denote
\begin{equation}\label{Theta, Omega}
\begin{cases}
    U:= \sup_{\nu>0} \|u^\nu\|_{L^\infty_t L^2_x} +  \|u\|_{L^\infty_t L^2_x},\\
    \Omega_p := \sup_{\nu >0 }\|\omega^\nu\|_{L^\infty_t L^p_x}  + \|\omega\|_{L^\infty_t L^p_x}\qquad \text{for any } p \in [1,\infty],\\
    \Theta := \sup_{\nu >0 }\left\{\|\theta^\nu\|_{L^1_t W^{1,\infty}_x} + \|\theta^\nu\|_{L^2_t H^1_x} \right\} + \left\{\|\theta\|_{L^1_t W^{1,\infty}_x} + \|\theta\|_{L^2_t H^1_x}\right\}.
\end{cases}
\end{equation}
\begin{lemma}\label{lem: Omega and Theta}
The quantities $\Theta$ and $\Omega_p$ for any $p \in [1,\infty]$ are finite. Indeed, they are bounded by constants depending only on the initial data $\|u_0\|_{L^2_x}$, $\sup_{\nu >0 }\|u^\nu_0\|_{L^2_x}$, 
    $\|\omega_0\|_{L^\infty_x}$, $\sup_{\nu >0 }\|\omega^\nu_0\|_{L^\infty_x}$, $\|\theta_0\|_{\left(\bes^1_{\infty,1}\right)_x \cap H^1_x}$, and  $\sup_{\nu >0 }\|\theta^\nu_0\|_{\left(\bes^1_{\infty,1}\right)_x \cap H^1_x}$.
\end{lemma}

\begin{proof}
    The finiteness of $U$ follows from simple energy considerations:
    \begin{align*}
        &\|u(t)\|^2_{L^2_x} \leq \|u_0\|^2_{L^2_x} e^t + \|\theta_0\|^2_{L^2_x}(e^t-1),\\
        &\|u^\nu(t)\|^2_{L^2_x} \leq \|u^\nu_0\|^2_{L^2_x} e^t + \|\theta^\nu_0\|^2_{L^2_x}(e^t-1).
    \end{align*}
    The finiteness of $\Theta$ is established in Danchin--Paicu \cite{DP}. The finiteness of $\Omega_\infty$ follows from 
\begin{align*}
&\|\omega^\nu(t)\|_{L^\infty_x} \leq \|\omega_0^\nu\|_{L^\infty_x} + \int_0^t \left\|\p_1\theta^\nu(s)\right\|_{L_x^\infty}\,\dd s,\\ &\|\omega(t)\|_{L^\infty_x} \leq \|\omega_0\|_{L^\infty_x} + \int_0^t \left\|\p_1\theta(s)\right\|_{L_x^\infty}\,\dd s,
\end{align*}
thanks to the transport structure of the vorticity equations~\eqref{vort eq, nu} and \eqref{vort eq, 0}, as well as $\Theta<\infty$. Finally, $\Omega_p \leq \Omega_\infty$ for any $p \in [1,\infty[$ by H\"{o}lder, since $\area\brator=1$.  \end{proof} 

With Lemma~\ref{lem: Omega and Theta} at hand, in the sequel, we fix the constant $\beta$ in Lemma~\ref{lem: MT}:
\begin{align}
    \beta = \frac{\gamma}{\Omega_\infty}.
\end{align}
We shall also use standard inequalities (Sobolev, Ladyszhenskalya, H\"{o}lder, and Cauchy--Schwarz) on $\tor$ in the remaining parts of the paper.

\section{Two key estimates}\label{sec: estimates}

This section is devoted to the proof of two crucial estimates, Propositions~\ref{propn-lem2} and \ref{propn-lem4} below. They are modelled over Lemmata~2 $\&$ 4 in Constantin--Drivas--Elgindi \cite{cde}, and the variants and slight improvements presented in our work are necessary for treating the Yudovich solutions to the Boussinesq equations.

\subsection{Quantified loss of $L^p$-regularity for scalars transported by Yudovich velocity}

The following result is a variant of \cite[Lemma~2]{cde}, which relaxes the requirement of regularity on forcing $f$ from $L^\infty_t L^\infty_x$ to $L^1_t L^\infty_x$. This is crucial to our development, as $\na \theta \in L^1_t L^\infty_x$ for the Yudovich solution to the Boussinesq equations. A similar idea for the proof of \cite[Lemma~2]{cde} was used in Elgindi--Jeong \cite{ej}.

\begin{proposition}\label{propn-lem2}
Let $u=u(t,x)$ be divergence-free, $\omega := {\rm rot} \,u \in L^\infty_t L^\infty_x$, and $\|\omega\|_{L^\infty_tL^\infty_x} \leq \Omega_\infty < \infty$. Suppose that a non-negative scalar field $\sigma = \sigma(t,x)$ satisfies the differential inequality
\begin{align*}
    \p_t\sigma + u\cdot\na \sigma - \nu \Delta \sigma \leq |\na u|\sigma + f,
\end{align*}
where $\sigma_0 := \sigma\big|_{t=0} \in L^\infty_x$ and $f \in L^1_t L^\infty_x$. Then
\begin{align*}
    \|\sigma\|_{L^\infty_t L^2_x\left(\left[0, \frac{\gamma(p-1)}{2p \Omega_\infty}\right] \times \tor\right)} \leq C\left(1+ \|\sigma_0\|^p_{L^{2p}_x} \right) \qquad\text{for any $p \in ]1,\infty[$},
\end{align*}
where $C$ depends only on $p$, $\gamma/\Omega_\infty$, and $\|f\|_{L^1_t L^\infty_x}$. (Recall the constants $\gamma$
and $\Omega_\infty$ from Lemmata~\ref{lem: MT} and \ref{lem: Omega and Theta}, respectively.)  The same result holds also for $\nu=0$.
\end{proposition}

This result states that the $L^2_x$-norm of a scalar transported by a Yudovich velocity field with $L^1_tL^\infty_x$-forcing can be controlled by the $L^{2+\e}_x$-norm of its initial datum uniformly in some time interval, which degenerates as $\e \to 0^+$. 

\begin{proof}
We essentially follow the arguments in \cite{cde}. Let us only emphasise the  modifications to the proof of Lemma~2 in \cite[pp.67--69]{cde}.

Fix $p_0 \in ]1,\infty[$. Consider a $C^1$-function $p=p(t)$  to be specified later that takes values in $[1,p_0]$. Energy estimate applied to $\|\theta(t)\|^{2p(t)}_{L^{2p(t)}_x}$ leads to the following:
\begin{align*}
   \frac{1}{2}\frac{\dd}{\dd t} \left(\|\sigma(t)\|^{2p(t)}_{L^{2p(t)}_x}\right) &\leq p'(t) \int_\tor \left(\log|\sigma|\right) |\sigma|^{2p(t)}\,\dd x\\
   &\quad + p(t) \int_\tor |\sigma|^{2p(t)}|\na u|\,\dd x + p(t) \|f(t)\|_{L^\infty_x} \|\sigma\|_{L^{2p(t)}_x}^{2p(t)}.
\end{align*}
See \cite[Estimate~(2.20)]{cde} for details. Here we bound
\begin{align*}
    \int_\tor |\sigma(t,x)|^{2p(t)-2}\sigma(t,x) f(t,x)\,\dd x \leq \|f(t)\|_{L^\infty_x}\|\sigma\|_{L^{2p(t)}_x}^{2p(t)-1},
\end{align*}
thanks to ${\rm Area}\brator =1$. In addition, the diffusion term is treated by
\begin{align*}
&\nu \int_\tor |\sigma(t,x)|^{2p(t)-2}\sigma(t,x)\Delta\sigma(t,x)\,\dd x\\
&\qquad= - \nu (2p(t)-1) \int_\tor |\sigma(t,x)|^{2p-2}|\na\sigma(t,x)|^2\,\dd x \leq 0,
\end{align*}
\emph{i.e.}, we only need the non-positivity of this term. Hence, our proof remains valid $\nu=0$.

An application of Lemma~\ref{lem: arithmetic} with $a = \beta |\na u|$, $b = \beta^{-1}|\sigma|^{2p(t)}$ therein yields the following slight variant of \cite[Estimate~(2.26)]{cde}:
\begin{align*}
\frac{1}{2}\frac{\dd}{\dd t} \left(\|\sigma(t)\|^{2p(t)}_{L^{2p(t)}_x}\right) &\leq \left(p'(t)  + \frac{2p^2(t)}{\beta} \right)\int_\tor \left(\log|\sigma|\right) |\sigma|^{2p(t)}\,\dd x\\
&\qquad+ \frac{p(t)}{\beta} \log(\beta^{-1}) \|\sigma(t)\|_{L^{2p(t)}_x}^{2p(t)}\\
&\qquad + p(t)C_K + Cp(t) \|f(t)\|_{L^\infty_x} \|\sigma(t)\|_{L^{2p(t)}_x}^{2p(t)-1}.
\end{align*}

As in \cite{cde}, choose $p=p(t)$ so that $p'(t)  + \frac{2p^2(t)}{\beta}=0$ and $p(0)=p_0$. That is, $$p(t) = \frac{\beta p_0}{\beta+2p_0 t},$$ which is a monotone decreasing function in $t$. Note that $p(t_\star)=1$ where $$t_\star = \frac{\beta(p_0-1)}{2p_0}.$$ Recall $\beta := \gamma/\Omega_\infty$ as before. The quantity
\begin{align}\label{def, m}
    m(t) := \|\sigma(t)\|^{2p(t)}_{L^{2p(t)}_x}
\end{align}
then satisfies the differential inequality:
\begin{align}\label{diff ineq}
    m'(t) \leq C p(t) \left[ \|f(t)\|_{L^\infty_x} \cdot m(t)^{\frac{2p(t)-1}{2p(t)}}  + m(t) + 1\right]
\end{align}
for some $C=C(C_K, \beta)$. Recall that $C_K$ is a universal constant; \emph{e.g.}, $C_K = 2 {\rm Area}\brator=2$.

To proceed, by Young's inequality, one has $$m(t)^{\frac{2p(t)-1}{2p(t)}} \leq m(t) \cdot \frac{2p(t)-1}{2p(t)} + \frac{1}{2p(t)}.$$ Since $p(t)\geq 1$, we deduce from \eqref{diff ineq} that
\begin{align*}
m'(t) \leq Cp(t) h(t)\big[m(t)+1\big],
\end{align*}
where $$h(t):= 1+ \|f(t)\|_{L^\infty_x}.$$

The above differential inequality can be easily solved by
\begin{align*}
    \log\big(m(t)+1\big) \leq \log\big(m(0)+1\big) + C \int_0^t p(s)h(s)\,\dd s,
\end{align*}
whence
\begin{align*}
    \int_0^t p(s)h(s)\,\dd s &= \int_0^t \frac{\beta p_0}{\beta + 2p_0 s} \left(1+ \|f(s)\|_{L^\infty_x}\right)\,\dd s\\
    &\leq p_0 \left(t_\star + \|f\|_{L^1_t L^\infty_x}\right).
\end{align*}
By setting $C':=e^{Cp_0 \left(t_\star + \|f\|_{L^1_t L^\infty_x}\right)}$ we thus obtain the bound:
\begin{align*}
  \sup_{t \in [0,t_\star]}  m(t) \leq C'\big(m(0)+1\big).
\end{align*}
Notice that $C'$ depends only on $\beta$, $C_K$, $p_0$, and $\|f\|_{L^1_t L^\infty_x}$. In the remaining parts of the proof, the constants $C'$ may change from line to line, but they retain the same dependence on parameters.

To conclude, we recall the definition of $m(t)$ in \eqref{def, m} to find that
\begin{align*}
\|\sigma(t)\|^{2p(t)}_{L^{2p(t)}_x} \leq C'\left\{\|\sigma(0)\|^{2p_0}_{L^{2p_0}_x}+1\right\} \qquad \text{for any } t \in [0,t_\star].
\end{align*}
Taking $\frac{1}{2p(t)}$-th power on both sides and noticing $\|\sigma(t)\|_{L^2_x} \leq \|\sigma(t)\|_{L^{2p(t)}_x}$ as $p(t) \geq 1$, we have
\begin{align*}
\|\sigma(t)\|_{L^{2}_x} \leq C'\left\{\|\sigma(0)\|_{L^{2p_0}_x}^{p_0/p(t)}+1\right\}.
\end{align*}
But here $p_0/p(t)$ is monotone increasing from $1$ to $p_0$ as $t$ varies from $0$ to $t_\star$; hence,
\begin{align*}
\|\sigma(t)\|_{L^{2}_x} \leq C'\left\{\|\sigma(0)\|^{p_0}_{L^{2p_0}_x}+1\right\}.
\end{align*}
We complete the proof by relabelling $p_0$ as $p$.    
\end{proof}

\subsection{Propagation of smallness for $u^\nu-u$}

The following variant of \cite[Lemma~4]{cde} states that,  for small inverse Reynolds number and small $\|(u^\nu-u)(t_0)\|_{L^2_x}$ at time $t_0$, for $t-t_0 \ll 1$ (quantified in terms of $\Omega_\infty$ and $\gamma$), the norm $\|(u^\nu-u)(t)\|_{L^2_x}$ remains small, with an effective bound on the propagation of smallness in terms of the physical parameters appearing in Lemma~\ref{lem: Omega and Theta}.

\begin{proposition}\label{propn-lem4}
Let $u^\nu$ and $u$ be the Yudovich solutions to the Boussinesq equations~\eqref{eq: bkn} and \eqref{eq: bk0}, respectively; that is, consider the initial data with uniformly bounded vorticities $\omega_0, \omega^\nu_0 \in L^\infty_x$.  
There exist constants $U$, $\Omega_2$, $\Theta$, and $\mathcal{K}$ depending only on the initial data such that the velocity difference $v:= u^\nu-u$ satisfies
\begin{align*}
    \frac{\|v(t)\|^2_{L^2}}{U^2} \leq 4 {\mathcal{K}}^{\frac{4 (t-t_0)\Omega_\infty}{\gamma}} \left\{\frac{\|v(t_0)\|^2_{L^2}}{U^2} + \frac{\gamma}{\Omega_\infty} \frac{\Omega_2^2}{U^2} \nu\right\}^{1-{\frac{20 (t-t_0)\Omega_\infty}{\gamma}}}
\end{align*}
for any $t \in [t_0,T]$, provided that $\nu$ and $\|v(t_0)\|$ can be chosen sufficiently small.
\end{proposition}

Indeed, the following choice of  $\mathcal{K}$ will be clear from the proof below:
\begin{align*}
\mathcal{K} = 2C_0\frac{\gamma}{\Omega_\infty} \left(
 \frac{\Omega_2\Omega_4}{U} + \sqrt{U}\Theta   \right)^4,
\end{align*}
where $C_0$ is a universal constant, $\gamma$ is the universal constant appearing in the Moser--Trudinger Lemma~\ref{lem: MT}, and $U$, $\Theta$, $\Omega_2$, $\Omega_4$, and $\Omega_\infty$ are as in \eqref{Theta, Omega} above.

Throughout the proof, $C$ denotes universal constants that may change from line to line. 

\begin{proof}
    We divide the arguments  into six steps below.

\smallskip
\noindent
{\bf Step~1.} By subtracting the velocity equations \eqref{eq: bkn} and \eqref{eq: bk0}, we deduce the PDE:
\begin{equation}\label{diff eq}
    \p_tv + u^\nu \cdot \na v + v \cdot\na u + \na p = \nu \Delta v + \nu \Delta u + \zeta \etwo,
\end{equation}
where 
\begin{align*}
    v:= u^\nu-u,\qquad \zeta := \theta^\nu-\theta,\qquad p := P^\nu-P.
\end{align*}
Standard energy estimates yield that
\begin{align}\label{energy est, lem4}
&\frac{\dd}{\dd t}\|v(t)\|^2_{L^2_x} + \nu \|\na v(t)\|^2_{L^2_x} \nonumber\\
&\qquad\leq \nu \|\na u(t)\|^2_{L^2_x} +2 \int_\tor |\na u(t,x)||v(t,x)|^2\,\dd x + 2\int_\tor|\zeta(t,x)||v(t,x)|\,\dd x.
\end{align}

In what follows, we recall $U$, $\Omega_p$, and $\Theta$ from Lemma~\ref{lem: Omega and Theta}.  We proceed by splitting the last two integrals on the right-hand side of \eqref{energy est, lem4} as $\int_\tor = \int_\bigt+\int_\smallt$, where for each $t \in [0,T]$ one defines the big and small parts, respectively:
\begin{align*}
    \bigt := \left\{x \in \tor:\, |v(t,x)| \geq MU\right\},\qquad \smallt:=\tor \setminus \bigt.
\end{align*}
The constant $M>0$ is to be specified. Note by the Markov inequality that
\begin{equation}\label{area estimate}
    \area\left(\bigt\right) \leq \frac{1}{M^2U^2} \int_\bigt |v(t,x)|^2\,\dd x.
\end{equation}

\smallskip
\noindent
{\bf Step~2.} In this step, we estimate the integrals over $\bigt$. Indeed, we have
\begin{align}\label{uv, BIG}
    \int_\bigt |\na u(t,x)||v(t,x)|^2\,\dd x &\leq 2 \|\na u(t)\|_{L^2_x(\bigt)} \|v(t)\|_{L^4_x(\bigt)}^2 \nonumber\\
    &\leq 2  \|\na u(t)\|_{L^4_x(\bigt)} \left\{\area\left(\bigt\right)\right\}^{\frac{1}{4}}\|v(t)\|_{L^4_x(\bigt)}^2\nonumber\\
    &\leq C\frac{\Omega_4 \Omega_2 U}{\sqrt{M}}.
\end{align}
The first two lines above follow from the H\"{o}lder inequality, while the final line holds by the div-curl estimate $\|\na u\|_{L^p_x} \lesssim_p \|\omega\|_{L^p_x}$ for $p \in [1,\infty[$ and ${\rm div}\,u=0$,  the Ladyszhenskaya inequality $\|v\|_{L^4_x} \lesssim \|v\|_{L^2_x}^{1/2}\|\na v\|_{L^2_x}^{1/2}$, the definition of $U$, and the bound~\eqref{area estimate}. 

Similarly, we obtain that
\begin{align}\label{zeta v, BIG}
    \int_\bigt |\zeta(t,x)||v(t,x)|\,\dd x &\leq \|\zeta(t)\|_{L^4_x(\bigt)} \|v\|_{L^{4/3}_x(\bigt)}\nonumber\\
    &\leq \|\zeta(t)\|_{L^\infty_x} \left\{ \area\left(\bigt\right) \right\}^{1/4}\|v(t)\|_{L^{4/3}_x} \nonumber\\
    &\leq \|\zeta(t)\|_{L^\infty_x}  \frac{\|v(t)\|_{L^2_x}^{3/2}}{\sqrt{MU}} \nonumber\\
    &\leq \sqrt{\frac{U}{M}} \|\zeta(t)\|_{L^\infty_x},
\end{align}
thanks to H\"{o}lder's inequality and $\|v(t)\|_{L^{4/3}_x} \leq \|v(t)\|_{L^2_x}\area\brator^{1/4}$ with $\area\brator=1$.

\smallskip
\noindent
{\bf Step~3.} In this step, we estimate the integrals over $\smallt$. 

Recall the Moser--Trudinger estimate (Lemma~\ref{lem: MT}): $\int_\tor e^{\beta|\na u|}\,\dd x \leq C_K$ where $\beta = \gamma/\Omega_\infty$ and $C_K$ is a universal constant. Applying the arithmetic Lemma~\ref{lem: arithmetic} to $a =a(t,x) = \beta|\na u| + \log \e$ and $b = b(t,x) = |v|^2U^{-2}$ with $\e>0$ to be determined, we deduce that
\begin{align}\label{uv, SMALL}
&    \beta U^{-2} \int_\smallt |\na u(t,x)| |v(t,x)|^2\,\dd x\nonumber \\
&= \int_\smallt \Big(a(t,x)-\log\e\Big)b(t,x)\,\dd x\nonumber \\
    &\leq \e \int_\tor e^{\beta |\na u(t,x)|}\,\dd x + \int_\smallt \frac{|v(t,x)|^2}{U^2}\log\left(\frac{|v(t,x)|^2}{U^2}\right)\,\dd x   - \log\e\int_\smallt \frac{|v(t,x)|^2}{U^2}\,\dd x \nonumber\\
    &\leq \e C_K + \log\left(\frac{M^2}{\e}\right) y(t),
\end{align}
where 
\begin{equation}\label{y, def}
    y(t):= \frac{\|v(t)\|^2_{L^2_x}}{U^2}.
\end{equation}

On the other hand, by the Cauchy--Schwarz inequality, we have
\begin{align}\label{zeta v, SMALL}
    \beta U^{-2} \int_\smallt |\zeta(t,x)||v(t,x)|\,\dd x &\leq \frac{\beta}{U} \|\zeta(t)\|_{L^2_x} \,\sqrt{y(t)} \nonumber\\
    &\leq \Lambda y(t) + \frac{\beta^2}{4\Lambda U^2}\|\zeta(t)\|^2_{L^2_x},
\end{align}
where $\Lambda>0$ is a large (dimensionless) constant to be specified.

\smallskip
\noindent
{\bf Step~4.} Denote as in \cite{cde} the non-dimensionalised constant 
\begin{align*}
    \G := \frac{M^2}{\e}.
\end{align*}
In later steps, we shall require $\G\geq e^{1/4}$.

Substituting the bounds~\eqref{uv, BIG}, \eqref{zeta v, BIG}, \eqref{uv, SMALL}, and \eqref{zeta v, SMALL} into the energy estimate~\eqref{energy est, lem4}, we arrive at the differential inequality:
\begin{align}\label{diff ineq, 1}
    &\beta y'(t) + \beta \nu \frac{\|\na v(t)\|^2_{L^2}}{U^2} \nonumber\\
    &\qquad \leq F + 2(\Lambda + \log\G)y(t) +  \frac{\beta\sqrt{U}}{\sqrt{M}}\|\zeta(t)\|_{L^\infty_x} + \frac{\beta^2}{2U^2\Lambda}\|\zeta(t)\|^2_{L^2_x},
\end{align}
with the constant term
\begin{equation}\label{F, def}
    F := \beta\nu \frac{\Omega_2^2}{U^2} + C\frac{\beta \Omega_2\Omega_4}{\sqrt{M} U} + 2\e C_K.
\end{equation}
We emphasise that $C$ and $C_K$ are both universal constants. One may always enlarge $C$ without violating any of the inequalities here and hereafter.

\smallskip
\noindent
{\bf Step~5.} 
We apply the Gr\"{o}nwall  Lemma~\ref{lem: gronwall} with
\begin{equation*}
    \begin{cases}
        A(t) \equiv A = \frac{2}{\beta}(\Lambda+\log\G),\\
        B(t) := \frac{F}{\beta}+ \frac{\beta\sqrt{U}}{\sqrt{M}}\|\zeta(t)\|_{L^\infty_x} + \frac{\beta^2}{2U^2\Lambda}\|\zeta(t)\|^2_{L^2_x}.
    \end{cases}
\end{equation*}
Here $B \in L^1_t$ since $\theta$ and $\theta^\nu \in L^1_t \left(\bes^1_{\infty,1}\right)_x \cap L_t^2 L_x^2$ for Yudovich solutions. In the sequel, we shall select the parameters $\Lambda$, $M$, and $\e$ in this order.

We first choose 
\begin{align}\label{choice of Lambda}
    \Lambda = \log\G,
\end{align}
which leads to
\begin{equation*}
    \begin{cases}
        A(t) \equiv A = \frac{4}{\beta}\log\G,\\
        B(t) := \frac{F}{\beta}+ \frac{\beta\sqrt{U}}{\sqrt{M}}\|\zeta(t)\|_{L^\infty_x} + \frac{\beta^2}{2U^2\log\G}\|\zeta(t)\|^2_{L^2_x}.
    \end{cases}
\end{equation*}
One thus infers from Lemma~\ref{lem: gronwall} and \eqref{diff ineq, 1} that
\begin{align}\label{y, one}
    y(t) &\leq y(t_0) \G^{\frac{4(t-t_0)}{\beta}} + \int_{t_0}^t  
    \G^{\frac{4(t-s)}{\beta}}\left\{\frac{F}{\beta}+ \frac{\beta\sqrt{U}}{\sqrt{M}}\|\zeta(s)\|_{L^\infty_x} + \frac{\beta^2}{2U^2\log\G}\|\zeta(s)\|^2_{L^2_x}\right\}\,\dd s\nonumber\\
    &\leq y(t_0) \G^{\frac{4(t-t_0)}{\beta}} + \frac{F}{\beta}\int_{t_0}^t 
    \G^{\frac{4(t-s)}{\beta}}\,\dd s \nonumber  \\
    &\qquad\qquad + \G^{\frac{4(t-t_0)}{\beta}}\int_{t_0}^t\left\{\frac{\beta\sqrt{U}}{\sqrt{M}}\|\zeta(s)\|_{L^\infty_x} + \frac{\beta^2}{2U^2\log\G}\|\zeta(s)\|^2_{L^2_x}\right\}\,\dd s\nonumber \\
    &\leq \G^{\frac{4(t-t_0)}{\beta}} \left\{y(t_0) + \frac{\beta\nu \frac{\Omega_2^2}{U^2} + C\frac{\beta \Omega_2\Omega_4}{\sqrt{M} U} + 2\e C_K}{4\log\G} + \frac{\beta\sqrt{U}}{\sqrt{M}}\Theta +  \frac{\beta^2}{2U^2\log\G}\Theta^2\right\}\nonumber \\
    &\leq \G^{\frac{4(t-t_0)}{\beta}} \left\{y(t_0) + \beta\nu \frac{\Omega_2^2}{U^2} + \left[C\frac{\beta \Omega_2\Omega_4}{U} + \beta\sqrt{U}\Theta\right]\frac{1}{\sqrt{M}} + 2\e C_K   +  \frac{\beta^2\Theta^2}{2U^2\log\G}\right\}.
\end{align}
In the last two lines we use the expression \eqref{F, def} for $F$ and $\log \G \geq 1/4$ (justified later).

Denote
\begin{equation*}
    \eta(t_0):=y(t_0) + \beta\nu \frac{\Omega_2^2}{U^2}.
\end{equation*}
This quantity can be made arbitrarily small if $\nu$ and $\|v(t_0)\|_{L^2_x}$ are small. We next choose $M \gg 1$ such that
\begin{equation}\label{choice of M}
\left[C\frac{\beta \Omega_2\Omega_4}{U} + \beta\sqrt{U}\Theta\right]\frac{1}{\sqrt{M}}= \eta(t_0),
\end{equation}
and finally choose $\e \ll 1$ such that
\begin{equation}\label{choice of epsilon}
    \max\left\{2\e C_K, \frac{\beta^2\Theta^2}{ U^2\log\left(\frac{M^2}{\e}\right)} \right\} =  \eta(t_0).
\end{equation}
For \eqref{choice of epsilon}, recall $\G = M^2/\e$ and that $M$ therein has already been fixed in \eqref{choice of M} above.

Substituting \eqref{choice of epsilon} and \eqref{choice of M} into the final line of \eqref{y, one}, we arrive at
\begin{align}\label{y, two}
    y(t) \leq 4\G^{\frac{4(t-t_0)}{\beta}}  \eta(t_0).
\end{align}
Meanwhile, using once again \eqref{choice of M} and \eqref{choice of epsilon}, we deduce that
\begin{align*}
    \G = \frac{M^2}{\e} \leq \frac{\left\{ \frac{C\beta \Omega_2\Omega_4}{U\eta(t_0)} \right\}^4}{ \frac{\eta(t_0)}{2C_K}} = \frac{\mathcal{K}}{\eta(t_0)^5},
\end{align*}
where $\mathcal{K} = 2C_K \left(
 \frac{C\beta\Omega_2\Omega_4}{U} + \beta\sqrt{U}\Theta   \right)^4$.

 Therefore, we obtain from \eqref{y, two} that
\begin{align*}
    y(t) \leq 4{\mathcal{K} }^{\frac{4(t-t_0)}{\beta}}\eta(t_0)^{1-\frac{20(t-t_0)}{\beta}},
\end{align*}
which is the desired estimate. Note that $\G$ can be made large (say, $\G > e^{1/4}$) when $\eta(t_0)$ is sufficiently small. This completes the proof.   \end{proof}

\section{Conclusion}\label{sec: conclusion}

We are now ready to conclude the proof of the Main Theorem~\ref{thm: main}. The arguments are an adaptation of Constantin--Drivas--Elgindi \cite[Proof of Proposition~2, Step~2 in the proof of Lemma~4, Lemma~3, and Proof of Proposition~1]{cde}. Modifications in several places are needed to take care of the forcing terms $\theta^\nu$ and $\theta \in L^1_tL^\infty_x$.

\begin{proof}
The arguments are divided into four steps.

\smallskip
\noindent
{\bf Step~1.} We first show that it suffices to establish the $L^\infty_tL^2_x$-inviscid limit of vorticities for a mollified problem. Indeed, let $\varphi(x)$ be the standard mollifier on $\tor$ and $\varphi_\ell:=\ell^{-2}\varphi(\ell^2x)$. Consider the regularised vorticity equations (derived from Eqs.~\eqref{vort eq, nu} and \eqref{vort eq, 0}):
\begin{equation}\label{mollified vort eq, nu}
    \begin{cases}
        \p_t\omega_\ell^\nu + u^\nu\cdot\na \omega^\nu_\ell - \nu \Delta \omega^\nu_\ell = \varphi_\ell\star \p_1 \theta^\nu \qquad \text{in } [0,T] \times \tor,\\
        \omega_\ell^\nu\big|_{t=0} = \varphi_\ell\star \omega^\nu_0 \qquad \text{ at } \{0\} \times \tor,
    \end{cases}
\end{equation}
as well as \begin{equation}\label{mollified vort eq, 0}
    \begin{cases}
        \p_t\omega_\ell + u\cdot\na \omega_\ell = \varphi_\ell\star \p_1 \theta \qquad \text{in } [0,T] \times \tor,\\
        \omega_\ell\big|_{t=0} = \varphi_\ell\star \omega_0 \qquad \text{ at } \{0\} \times \tor,
    \end{cases}
\end{equation}
where $(u,\theta)$ and $(u^\nu,\theta^\nu)$ are the Yudovich solutions to the Euler--Boussinesq Equation~\eqref{eq: bk0} and the Boussinesq Equation~\eqref{eq: bkn}, respectively. Since
\begin{align*}
&\|\omega_\ell(t) - \omega(t)\|_{L^2_x} \leq \|\omega_0 -\varphi_\ell\star \omega_0\|_{L^2_x} + \int_0^t \| \p_1\theta(s) - \varphi_\ell \star \p_1\theta(s)\|_{L^2_x}\,\dd s,\\
&\|\omega^\nu_\ell(t) - \omega^\nu(t)\|_{L^2_x} \leq \|\omega^\nu_0 -\varphi_\ell\star \omega^\nu_0\|_{L^2_x} + \int_0^t \| \p_1\theta^\nu(s) - \varphi_\ell \star \p_1\theta^\nu(s)\|_{L^2_x}\,\dd s,
\end{align*}
these terms tend to zero as $\ell \searrow 0$ by the estimates for Yudovich solutions in Theorem~\ref{thm: DP}. Hence, by triangle inequality, interpolation $L^2_x \cap L^\infty_x \emb L^p_x$ for any $p > 2$, and the uniform boundedness of $L^\infty_x$-norm of $\omega^\nu$ and $\omega$ (namely, $\Theta<\infty$ in Lemma~\ref{lem: Omega and Theta}), it suffices to prove that
\begin{equation}\label{suff}
    \lim_{\nu \searrow 0} \left\|\omega^\nu_\ell - \omega_\ell\right\|_{L^\infty_t L^2_x} = 0
\end{equation}
for each $\ell \in \mathbb{N}$.

\smallskip
\noindent
{\bf Step~2.} 
We now prove that, for each \emph{fixed} $\ell \in \mathbb{N}$, on $[0,T_\star]$ for some $T_\star \lesssim \Omega_\infty^{-1}$ modulo a universal constant (in particular, independent of $\ell$),
\begin{equation}\label{H1-vort}
\sup_{\nu >0} \|\omega^\nu_\ell\|_{L^\infty_t H^1_x} + \|\omega_\ell\|_{L^\infty_t H^1_x} \leq C_\ell.
\end{equation}
Here the constant $C_\ell$ depends only on $\ell$.

For this purpose, note as on \cite[p.69]{cde} that the non-negative scalar field $\sigma = \left|\na \omega^\nu_\ell\right|$ satisfies the pointwise differential inequality:
\begin{align*}
    \left(\p_t + u^\nu\cdot\na - \nu \Delta\right) \sigma \leq |\na u|\sigma + \left|\na \left(\varphi_\ell \star \p_1 \theta^\nu\right)\right|
    \end{align*}
Here, the forcing term 
\begin{align*}
\left\|\na \left(\varphi_\ell \star \p_1 \theta^\nu\right)\right\|_{L^1_tL^\infty_x} \lesssim \ell^{-1}\left\|\p_1 \theta^\nu\right\|_{L^1_tL^\infty_x} \lesssim \ell^{-1}\Theta
\end{align*}
modulo universal constants, with $\Theta$ as in Lemma~\ref{lem: Omega and Theta}.\footnote{In fact, we have the improved bound $\left\|\na \left(\varphi_\ell \star \p_1 \theta^\nu\right)\right\|_{L^1_tL^\infty_x} = \left\|\left(\na\p_1\varphi_\ell\right)\star\theta^\nu\right\|_{L^1_tL^\infty_x} \lesssim \ell^{-2}\left\| \theta^\nu\right\|_{L^1_tL^\infty_x}$ here.}  Hence, we infer from Proposition~\ref{propn-lem2} that 
\begin{equation}\label{epsilon}
\left\|\omega_\ell^\nu\right\|_{L^\infty_t H^1_x\left(\left[0, \frac{\gamma\e}{2(1+\e)\Omega_\infty} \right] \times \tor\right)} \leq C_\e  \left(1+ \left\|\omega_0^\nu \star \na\varphi_\ell\right\|^{1+\e}_{L^{2(1+\e)}_x}\right)\quad\text{for arbitrary } \e>0,
\end{equation}
where $C_\e$ depends only on $\e$, $\beta=\gamma/\Omega_\infty$, $\ell$, and $\Theta$. It is clear that the term $\left\|\omega_0^\nu \star \na\varphi_\ell\right\|^{1+\e}_{L^{2(1+\e)}_x}$ is bounded by a constant depending only on $\Omega_\infty$, $\e$, and $\ell$. We fix the choice of $\e>0$ once and for all; to fix the idea, let us take $\e=1$.

All the above arguments remain valid for $\nu=0$, and all the constants are uniform in $\nu$. Thus, \eqref{H1-vort} follows from \eqref{epsilon}.

\smallskip
\noindent
{\bf Step~3.} From the mollified vorticity equations~\eqref{mollified vort eq, nu}, \eqref{mollified vort eq, 0}, Cauchy--Schwarz, and the bound~\eqref{H1-vort} in Step~2 above, we deduce that
\begin{align}\label{vort diff, energy}
   \frac{1}{2}\frac{\dd}{\dd t} \|\omega^\nu_\ell(t)-\omega_\ell(t)\|^2_{L^2_x} &= - \int_\tor (u^\nu-u) \cdot \na \omega^\nu_\ell \left(\omega^\nu_\ell-\omega_\ell\right)\,\dd x\nonumber \\
   &\qquad - \nu \int_\tor \left|\na \omega^\nu_\ell\right|^2\,\dd x + \nu \int_\tor \na\omega^\nu_\ell \cdot \na\omega_\ell\,\dd x \nonumber\\
   &\qquad + \int_\tor \left(\omega^\nu_\ell-\omega_\ell\right)\left[\varphi_\ell\star\p_1(\theta^\nu-\theta)\right]\,\dd x\nonumber\\
   &\leq C_\ell\Big\{\|u^\nu(t)-u(t)\|_{L^2_x} + \|\theta^\nu(t)-\theta(t)\|_{L^2_x}\Big\}\|\omega^\nu_\ell(t)-\omega_\ell(t)\|_{L^2_x}  + C_\ell\nu,
\end{align}
where $C_\ell$ depends on $\ell$, $\Omega_\infty$, and $\Theta$. Here we make use of the simple observation $$\varphi_\ell\star\p_1(\theta^\nu-\theta)=(\p_1\varphi_\ell)\star(\theta^\nu-\theta).$$

By the Boussinesq equations~\eqref{eq: bkn}, \eqref{eq: bk0}, $\theta^\nu-\theta$ satisfies the PDE:
\begin{align*}
    \p_t (\theta^\nu-\theta) + (u^\nu-u)\cdot\na \theta^\nu + u \cdot \na (\theta^\nu-\theta) - \kappa \Delta(\theta^\nu-\theta) = 0.
\end{align*}
By a standard energy estimate, we have
\begin{align*}
\frac{\dd}{\dd t} \|(\theta^\nu-\theta)(t)\|^2_{L^2_x} \leq 2\|\na \theta^\nu(t)\|_{L^\infty_x} \|(u^\nu-u)(t)\|_{L^2_x} \|(\theta^\nu-\theta)(t)\|_{L^2_x}.
\end{align*}
Thus, by the Gr\"{o}nwall Lemma~\ref{lem: gronwall} and the bounds in Lemma~\ref{lem: Omega and Theta},
\begin{align*}
    \|(\theta^\nu-\theta)(t)\|^2_{L^2_x} &\leq \|\theta^\nu_0-\theta_0\|^2_{L^2_x} \cdot \exp\left\{{2\int_0^t}\|\na \theta^\nu(s)\|_{L^\infty_x} \|(u^\nu-u)(s)\|_{L^2_x}\,\dd s\right\}\nonumber\\
    &\leq e^{2U\Theta}\|\theta^\nu_0-\theta_0\|^2_{L^2_x}.
\end{align*}
Here we use that $\|u^\nu\|_{L^2_x} + \|u\|_{L^2_x} \leq U$ and $\|\theta^\nu\|_{L^1_tW^{1,\infty}_x} \leq \Theta$. This bound for $\theta^\nu-\theta$ holds for any $t \in [0,T]$, not just on $\left[0,\frac{\gamma}{4\Omega_\infty}\right]$.

Therefore, in view of \eqref{vort diff, energy} and the Gr\"{o}nwall Lemma~\ref{lem: gronwall}, it remains to prove the strong convergence of velocity:
\begin{equation}\label{velocity convergence}
    \lim_{\nu \searrow 0} \left\|u^\nu-u\right\|_{L^\infty_t L^2_x\left(\left[0,\frac{\gamma}{4\Omega_\infty}\right]\times\tor\right)} = 0. 
\end{equation}
Indeed, assuming \eqref{velocity convergence}, we may conclude that $\lim_{\nu \searrow 0} \|\omega^\nu_\ell-\omega_\ell\|_{L^\infty_t L^2_x\left(\left[0,\frac{\gamma}{4\Omega_\infty}\right]\times\tor\right)} = 0$. Then, for an arbitrary $T>0$, we argue on $[0,T_\star]$, $[T_\star, 2T_\star]$,  $[2T_\star, 3T_\star]$, $\ldots$ one by one; $T_\star = \frac{\gamma}{4\Omega_\infty}$. On each time interval we take the terminal data from the previous interval as the initial data. In this way, we obtain \eqref{suff} in Step~1 of the same proof, which suffices to conclude Theorem~\ref{thm: main}.

\smallskip
\noindent
{\bf Step~4.} The proof of \eqref{velocity convergence} follows \cite[Step~2, p.74]{cde} almost verbatim. We sketch it here only for the sake of completeness.

Recall from Proposition~\ref{propn-lem4} that \begin{align*}
    \frac{\|v(t)\|^2_{L^2}}{U^2} \leq 4 {\mathcal{K}}^{\frac{4 (t-t_0)\Omega_\infty}{\gamma}} \left\{\frac{\|v(t_0)\|^2_{L^2}}{U^2} + \frac{\gamma}{\Omega_\infty} \frac{\Omega_2^2}{U^2} \nu\right\}^{1-{\frac{20 (t-t_0)\Omega_\infty}{\gamma}}},
\end{align*}
where $v := u^\nu-u$. Discretise the time interval $[0,t]$ by $t_j=t_{j-1}+\Delta t$, $j \in \{1,2,\ldots,n\}$, where 
\begin{equation*}
 \Delta t = \frac{1}{2c},\qquad c \equiv \frac{20\Omega_\infty}{\gamma}.
\end{equation*}
Then, $y_j \equiv y(t_j) \equiv  {\|v(t_j)\|^2_{L^2}}{U^{-2}}$ satisfies
\begin{align*}
    y_j \leq C_1\sqrt{y_{j-1}+ C_2\nu}\qquad \text{for each }j \in \{1,2,\ldots,n\},
\end{align*}
where $C_1 = 4\mathcal{K}^{40}$ and $C_2 = \frac{\gamma}{\Omega_\infty} \frac{\Omega_2^2}{U^2}$. Solving this iterative inequality, the quantities
\begin{align*}
    \delta_n = \frac{y_n+C_2\nu}{C_1^2},\qquad \tnu = \frac{C_2\nu}{C_1^2}
\end{align*}
satisfy 
\begin{align}\label{x}
    \delta_n \leq (\delta_0)^{2^{-n}} + \frac{\tnu^{2^{-n+1}}}{1-\tnu}.
\end{align}
Moreover, whenever $\tnu \leq (\sqrt{5}-1)^{-1}$ and $\delta_0 \in ]0,r_\star[$ where $r_\star=r_\star(\tnu)$ is the positive root of $x^2-x-\tnu=0$, then $\sup_{n \in \mathbf{N}}\delta_n \leq r_\star$. Such smallness conditions are easily satisfied by choosing $\|u_0^\nu-u_0\|_{L^2_x}$ and $\nu$ suitably small, in terms of constants depending only on $C_1$ and $C_2$. 

To conclude, fix any $t>0$ and set 
\begin{align*}
    n = \left\lfloor \frac{t}{\Delta t}\right\rfloor + 1 = \left\lfloor \frac{40\Omega_\infty}{\gamma}t\right\rfloor + 1.
\end{align*}
Since $x \mapsto x^{2^{-n}}$ is subadditive, we deduce from \eqref{x} that
\begin{align*}
    \|(u^\nu-u)(t)\|_{L^2_x}^2 &\leq C_1^2 U^2 \left\{ \frac{\|u^\nu_0-u_0\|_{L^2_x}^2}{C_1^2U^2} + 2 \frac{C_2}{C_1^2}\nu \right\}^{2^{\frac{-40\Omega_\infty}{\gamma}t}}\\
    & = \left(16\mathcal{K}^{80}U^2\right)^{1-e^{\frac{-40 \log 2\cdot \Omega_\infty}{\gamma}t}} \left\{\|u^\nu_0-u_0\|_{L^2_x}^2 + 2C_2U^2\nu \right\}^{e^{\frac{-40 \log 2\cdot \Omega_\infty}{\gamma}t}}.
\end{align*}
The right-hand side converges to zero as $\nu \searrow 0$, which yields \eqref{velocity convergence}.

This completes the proof of Theorem~\ref{thm: main}.    \end{proof}

\section{Comments}\label{sec: comments}

Several conclusion remarks are in order.

\begin{enumerate}
    \item 
As with the inviscid limit theorem established in \cite{cde} for Navier--Stokes, the results in this paper work only for the periodic domain $\tor$. For general bounded domains, Prandtl boundary layers will develop under the no-slip boundary condition; so, in general, the strong inviscid limit $u^\nu \to u$ in $\bigcap_{p \in [1,\infty[}L^\infty_t L^p_x$ fails near the boundary. However, the strong inviscid limit may remain valid for slip type boundary conditions. For example, under the Navier boundary condition, $u^\nu \to u$ at a rate $\mathcal{O}(\sqrt{\nu})$ in suitable weighted  Sobolev spaces. See Li--Wang \cite{mengni}, Bleitner--Carlson--Nobili \cite{bdry}, and others.

\item 
Motivated by Nguyen~\cite[Theorems~1.3 and 1.4]{nguyen} (which improves \cite[Corollary~1]{cde}), we expect that the solution semigroup $S_t: \linf\brator \to  \linf\brator$, $S_t(\omega_0) = \omega(t,\bullet)$ for the Euler--Boussinesq Equation~\eqref{eq: bk0} is $L^p \to L^p$ continuous for every $1 \leq p <\infty$ and $(\linf, w^\star) \to (\linf, w^\star)$ continuous, where $w^\star$ denotes the weak-star topology on $\linf\brator$.

    \item 
The 2D Boussinesq equations were first proposed as a  model for oceanographic and atmospheric fluid flows \cite{ped}, so it is of physical interest to study them on surfaces with non-Euclidean geometry. See Saito \cite{saito} and Li--Wu--Zhao \cite{lwz} for the analysis of Boussinesq equations on 2-sphere and general compact surfaces. We expect that the results in this paper remain valid for Boussinesq equations on compact surfaces without boundaries (with $\etwo$ replaced by any unit normal vector field in \eqref{eq: bkn}, \eqref{eq: bk0}).

\item 
It is also of mathematical and physical interest to investigate the inviscid limit for variants of 2D Boussinesq equations with partial dissipation, or with modified damping mechanisms. See, \emph{e.g.}, \cite{2,3,4,5,10,15,21,c8,c10,34,35,x,c11,c12,45,46,47,48,c2,49,59,61}. 
    
\end{enumerate}

\medskip

\noindent
{\bf Acknowledgement}. The research of SL is supported by NSFC Projects 12331008 $\&$ 12411530065, Young Elite Scientists Sponsorship Program by CAST 2023QNRC001, the National Key Research $\&$ Development Programs 2023YFA1010900 and 2024YFA1014900, Shanghai Rising-Star Program 24QA2703600, Shanghai Qi-Guang Scholarship, and Shanghai Frontiers Science Center of Modern Analysis.

\medskip
\noindent
{\bf Competing Interests Statement}. We declare that there are no conflicts of interest involved.

\medskip
\noindent
{\bf Data Availability Statement}. We declare that no data are associated with this work.


\begin{thebibliography}{99}

\bibitem{1}
H. Abidi and T. Hmidi, On the global well-posedness for Boussinesq system, \textit{J. Diff. Equ.} \textbf{233} (2007), 199--220.

\bibitem{c6}
D. Adhikari, O. Ben Said, U.~R. Pandey, and J. Wu, Stability and large-time behavior for the 2D Boussineq system with horizontal dissipation and vertical thermal diffusion, \textit{NoDEA Nonlinear Differential Equations Appl.} \textbf{29} (2022), Paper No. 42, 43 pp.

\bibitem{2}
D. Adhikari, C. Cao, H. Shang, J. Wu, X. Xu, and Z. Ye, Global regularity results for the 2D
Boussinesq equations with partial dissipation, \textit{J. Diff. Equ.} \textbf{260} (2016), 1893--1917.

\bibitem{3}
D. Adhikari, C. Cao, and J. Wu, The 2D Boussinesq equations with vertical viscosity and
vertical diffusivity, \textit{J. Diff. Equ.} \textbf{249} (2010), 1078--1088.


\bibitem{4}
D. Adhikari, C. Cao, and J. Wu, Global regularity results for the 2D Boussinesq equations
with vertical dissipation, \textit{J. Diff. Equ.} \textbf{251} (2011), 1637--1655.

\bibitem{5}
D. Adhikari, C. Cao, J. Wu, and X. Xu, Small global solutions to the damped two-dimensional
Boussinesq equations, \textit{J. Diff. Equ.} \textbf{256} (2014), 3594--3613.

\bibitem{c4}
O. Ben Said, U.~R. Pandey, and J. Wu, The stabilizing effect of the temperature on buoyancy-driven fluids, \textit{Indiana Univ. Math. J.} \textbf{71} (2022), 2605--2645.

\bibitem{10}
A. Biswas, C. Foias, and A. Larios, On the attractor for the semi-dissipative Boussinesq
equations, \textit{Ann. Inst. H. Poincaré Anal. Non Linéaire} \textbf{34} (2017), 381--405.

\bibitem{bdry}
F. Bleitner, E. Carlson, and C. Nobili, Large-time behavior of the 2D thermally non-diffusive Boussinesq equations with Navier-slip boundary conditions, \textit{Z. Angew. Math. Phys.} \textbf{76} (2025), Paper No. 58, 36 pp.

\bibitem{14}
J.R. Cannon and E. DiBenedetto, The initial value problem for the Boussinesq equations
with data in $L^p$, Approximation Methods for Navier-Stokes Problems (Proc. Sympos., Univ. Paderborn, Paderborn, 1979), \textit{Lecture Notes in Math.} \textbf{771}, Springer, Berlin, 1980, 129--144.

\bibitem{15}
C. Cao and J. Wu, Global regularity for the 2D anisotropic Boussinesq equations with vertical dissipation, \textit{Arch. Ration. Mech. Anal.} \textbf{208} (2013), 985--1004.


\bibitem{chae}
D. Chae, Global regularity for the 2D Boussinesq equations with partial viscosity terms, \textit{Adv.
Math.} \textbf{203} (2006), 497--513.

\bibitem{17}
D. Chae, P. Constantin, and J. Wu, An incompressible 2D didactic model with singularity and explicit solutions of the 2D Boussinesq equations, \textit{J. Math. Fluid Mech.} \textbf{16} (2014), 473--480.

\bibitem{18}
D. Chae and H. Nam, Local existence and blow-up criterion for the Boussinesq equations,
\textit{Proc. Roy. Soc. Edinburgh Sect. A} \textbf{127} (1997), 935--946.


\bibitem{19}
D. Chae, S. Kim, and H. Nam, Local existence and blow-up criterion of H\"{o}lder continuous solutions of the Boussinesq equations, \textit{Nagoya Math. J.} \textbf{155} (1999), 55--80.


\bibitem{20}
D. Chae and O.~Y. Imanuvilov, Generic solvability of the axisymmetric 3-D Euler equations
and the 2-D Boussinesq equations, \textit{J. Diff. Equ.} \textbf{156} (1999), 1--17.

\bibitem{21}
D. Chae and J. Wu, The 2D Boussinesq equations with logarithmically supercritical velocities, \textit{Adv. Math.} \textbf{230} (2012), 1618--1645.

\bibitem{c1}
J. Cheng, R. Ji, L. Tian, and J. Wu,  Stability and enhanced decay rate for 3D anisotropic Boussinesq equations near the hydrostatic balance, \textit{J. Differential Equations} \textbf{425} (2025), 300--341.

\bibitem{22}
D. Córdoba, C. Fefferman, and R. De La Llave, On squirt singularities in hydrodynamics,
 \textit{SIAM J. Math. Anal.} \textbf{36} (2004), 204--213.


\bibitem{cde}
P. Constantin, T.~D. Drivas, and T.~M. Elgindi, Inviscid limit of vorticity distributions in the Yudovich class, \textit{Comm. Pure Appl. Math.} \textbf{75} (2022), 60--82.


\bibitem{DP0}
R. Danchin and M. Paicu, Le th\'{e}or\`{e}me de Leray et le th\'{e}or\`{e}me de Fujita--Kato pour le syst\`{e}me de Boussinesq partiellement visqueux, \textit{Bull. Soc. Math. France} \textbf{136} (2008), 261--309. 

\bibitem{DP1}
R. Danchin and M. Paicu, Existence and uniqueness results for the Boussinesq system with
data in Lorentz spaces, \textit{Phys. D}, \textbf{237} (2008), 1444--1460. 


\bibitem{DP}
R. Danchin and M. Paicu, Global well-posedness issues for the inviscid Boussinesq system
with Yudovich’s type data, \textit{Comm. Math. Phys.} \textbf{290} (2009), 1--14.


\bibitem{DP2}
R. Danchin and M. Paicu, Global existence results for the anisotropic Boussinesq system in
dimension two, \textit{Math. Models Methods Appl. Sci.} \textbf{21} (2011), 421--457.

\bibitem{c8}
W. Deng, J. Wu, and P. Zhang,  Stability of Couette flow for 2D Boussinesq system with vertical dissipation, \textit{J. Funct. Anal.} \textbf{281} (2021),  Paper No. 109255, 40 pp.

\bibitem{dwzz}
C. Doering, J. Wu, K. Zhao, and X. Zheng, Long-time behavior of two-dimensional Boussinesq
equations without buoyancy diffusion, \textit{Phys. D} \textbf{376/377} (2018), 144--159.

\bibitem{c10}
B. Dong, J. Wu, X. Xu, and N. Zhu, Stability and exponential decay for the 2D anisotropic Boussinesq equations with horizontal dissipation, 
\textit{Calc. Var. Partial Differential Equations} \textbf{60} (2021), Paper No. 116, 21 pp.


\bibitem{27}
W. E and C.-W. Shu, Small-scale structures in Boussinesq convection, \textit{Phys. Fluids} \textbf{6} (1994),
49--58.

\bibitem{ej}
T.~M. Elgindi and I.-J. Jeong, Ill-posedness for the incompressible Euler equations in critical Sobolev spaces, \textit{Ann. PDE} \textbf{3} (2017), no.~1, Paper no.~7, 19pp.

\bibitem{guo}
B. Guo, Spectral method for solving two-dimensional Euler--Boussinesq equation, \textit{Acta Math. Appl. Sinica} \textbf{5} (1989), 27--50.

\bibitem{32}
T. Hmidi and S. Keraani, On the global well-posedness of the 2D Boussinesq system with a
zero diffusivity, \textit{Adv. Diff. Equ.} \textbf{12} (2007), 461--480.

\bibitem{33}
T. Hmidi and S. Keraani, On the global well-posedness of the Boussinesq system with zero
viscosity, \textit{Indiana Univ. Math. J.} \textbf{58} (2009), 1591--1618.

\bibitem{34}
T. Hmidi, S. Keraani, and F. Rousset, Global well-posedness for a Boussinesq-Navier-Stokes
system with critical dissipation, \textit{J. Diff. Equ.} \textbf{249} (2010), 2147--2174.

\bibitem{35}
T. Hmidi, S. Keraani, and F. Rousset, Global well-posedness for Euler-Boussinesq system with
critical dissipation, \textit{Comm. PDE} \textbf{36} (2011), 420--445.


\bibitem{36}
T. Hou and C. Li, Global well-posedness of the viscous Boussinesq equations, \textit{Disc. Cont.
Dyn. Sys.} \textbf{12} (2005), 1--12.


\bibitem{37}
L. Hu and H. Jian, Blow-up criterion for 2-D Boussinesq equations in bounded domain, \textit{Front. Math. China} \textbf{2} (2007), 559--581.


\bibitem{38}
W. Hu, I. Kukavica, and M. Ziane, Persistence of regularity for a viscous Boussinesq equations with zero diffusivity, \textit{Asymptot. Anal.} \textbf{91} (2015), 111--124.



\bibitem{39}
W. Hu, I. Kukavica, and M. Ziane, On the regularity for the Boussinesq equations in a bounded
domain, \textit{J. Math. Phys.} \textbf{54} (2013), 081507, 10 pp.

\bibitem{x}
R. Ji, D. Li, and J. Wu, Uniqueness of weak solutions to the Boussinesq equations with fractional dissipation, \textit{Commun. Math. Sci.} \textbf{21} (2023), 1531--1548.

\bibitem{c7}
R. Ji, Y. Li, and J. Wu, 
Optimal decay for the 3D anisotropic Boussinesq equations near the hydrostatic balance,
 \textit{Calc. Var. Partial Differential Equations} \textbf{61} (2022), Paper No. 136, 34 pp.

\bibitem{41}
Q. Jiu, C. Miao, J. Wu, and Z. Zhang, The 2D incompressible Boussinesq equations with
general critical dissipation, \textit{SIAM J. Math. Anal.} \textbf{46} (2014), 3426--3454.


\bibitem{42}
Q. Jiu, J. Wu, and W. Yang, Eventual regularity of the two-dimensional Boussinesq equations
with supercritical dissipation, \textit{J. Nonlinear Science} \textbf{25} (2015), 37--58.



\bibitem{44}
M. Lai, R. Pan, and K. Zhao, Initial boundary value problem for 2D viscous Boussinesq
equations, \textit{Arch. Ration. Mech. Anal.} \textbf{199} (2011), 739--760.

\bibitem{c11}
S. Lai, J. Wu, X. Xu, J. Zhang, and Y. Zhong, Optimal decay estimates for 2D Boussinesq equations with partial dissipation, \textit{J. Nonlinear Sci.} \textbf{31} (2021), Paper No. 16, 33 pp.


\bibitem{c12}
S. Lai, J. Wu, X. Xu, and Y. Zhong,  Stability and large-time behavior of the 2D Boussinesq equations with partial dissipation, 
\textit{J. Differential Equations} \textbf{271} (2021), 764--796.


\bibitem{45}
A. Larios, E. Lunasin, and E. S. Titi, Global well-posedness for the 2D Boussinesq system
with anisotropic viscosity and without heat diffusion, \textit{J. Diff. Equ.} \textbf{255} (2013), 2636--2654.


\bibitem{46}
D. Li and X. Xu, Global wellposedness of an inviscid 2D Boussinesq system with nonlinear
thermal diffusivity, \textit{Dyn. Par. Diff. Equ.} \textbf{10} (2013), 255–265.


\bibitem{47}
J. Li, H. Shang, J. Wu, X. Xu, and Z. Ye, Regularity criteria for the 2D Boussinesq equations
with supercritical dissipation, \textit{Comm. Math. Sci.} \textbf{14} (2016), 1999--2022.


\bibitem{48}
J. Li and E. S. Titi, Global well-posedness of the 2D Boussinesq equations with vertical dissipation, \textit{Arch. Ration. Mech. Anal.} \textbf{220} (2016), 983--1001.

\bibitem{mengni}
M. Li and Y.-L. Wang, Zero-viscosity limit for Boussinesq equations with vertical viscosity and Navier boundary in the half plane, \textit{Nonlinear Anal. Real World Appl.} \textbf{80} (2024), Paper No. 104150, 16 pp.

\bibitem{lwz}
S. Li, J. Wu, and K. Zhao, On the degenerate Boussinesq equations on surfaces, \textit{J. Geom. Mech.} \textbf{12} (2020), 107--140.

\bibitem{c2}
H. Lin, X. Suo, M. Jin, and J. Wu, Global well-posedness and decay estimates for the 3D anisotropic Boussinesq system with thermal damping, \textit{J. Differential Equations} \textbf{453} (2026), part 3, Paper No. 113863, 28 pp.

\bibitem{49}
S. A. Lorca and J. L. Boldrini, The initial value problem for a generalized Boussinesq model,
\textit{Nonlinear Analysis} \textbf{36} (1999), 457--480.


\bibitem{c5}
L. Ma, J. Wu, and Q. Zhang, Stability of 3D perturbations near a special 2D solution to the rotating Boussinesq equations, \textit{Stud. Appl. Math.} \textbf{148} (2022), 1624--1655.

\bibitem{mof}
H.~K. Moffatt, Some remarks on topological fluid mechanics, in: \textit{An Introduction to the Geometry and Topology of Fluid Flows}, R.~L. Ricca, ed., Dordrecht: Kluwer Academic Publishers, 2011, pp 3--10.

\bibitem{nguyen}
H.~Q. Nguyen, Remarks on the solution map for Yudovich solutions of the Euler equations, \textit{J. Math. Fluid Mech.} \textbf{24} (2022), Paper No. 44, 9 pp.

\bibitem{ped}
J. Pedlosky, \textit{Geophysical Fluid Dynamics}, New York: Pringer-Verlag, 1987.

\bibitem{51}
A. Sarria and J. Wu, Blowup in stagnation-point form solutions of the inviscid 2d Boussinesq
equations, \textit{J. Diff. Equ.} \textbf{259} (2015), 3559--3576.

\bibitem{saito}
J. Saito, Boussinesq equations in thin spherical domains, \textit{Kyushu J. Math.} \textbf{59} (2005), 443--465.

\bibitem{sww}
C. Seis, E. Wiedemann, and J. Woźnicki, Strong convergence of vorticities in the 2D viscosity limit on a bounded domain, \textit{J. Nonlinear Sci.} \textbf{36} (2026), Paper No. 22.

\bibitem{52}
A. Stefanov and J. Wu, A global regularity result for the 2D Boussinesq equations with critical dissipation, \textit{J. d’Analyse Mathematique} \textbf{137} (2019), 269--290.

\bibitem{b1}
A. Stefanov, J. Wu, X. Xu, and Z. Ye, Global regularity results of the 2D fractional Boussinesq equations, \textit{Math. Ann.} \textbf{391} (2025), 5965--6012.

\bibitem{53}
Y. Taniuchi, A note on the blow-up criterion for the inviscid 2-D Boussinesq equations, in: R.
Salvi (Ed.), {The Navier-Stokes Equations: Theory and Numerical Methods}, \textit{Lec. Notes Pure Appl. Math.} \textbf{223} (2002), 131--140.

\bibitem{54}
L. Tao, J. Wu, K. Zhao, and X. Zheng, Stability near hydrostatic equilibrium to the 2D Boussinesq equations without thermal diffusion, \textit{Arch. Ration. Mech. Anal.} \textbf{237} (2020), 585--630.


\bibitem{58}
J.~P. Whitehead and C.~R. Doering, Ultimate state of two-dimensional Rayleigh-Bénard
convection between free-slip fixed-temperature boundaries, \textit{Phys. Rev. Lett.} \textbf{106} (2011), 244501.

\bibitem{59}
J. Wu and X. Xu, Well-posedness and inviscid limits of the Boussinesq equations with fractional Laplacian dissipation, \textit{Nonlinearity} \textbf{27} (2014), 2215--2232.


\bibitem{60}
J. Wu, X. Xu, L. Xue, and Z. Ye, Regularity results for the 2D Boussinesq equations with
critical and supercritical dissipation, \textit{Comm. Math. Sci.} \textbf{14} (2016), 1963--1997.

\bibitem{c9}
J. Wu and Q. Zhang, Stability and optimal decay for a system of 3D anisotropic Boussinesq equations, \textit{Nonlinearity} \textbf{34} (2021), 5456--5484.

\bibitem{c3}
J. Wu and K. Zhao, On 2D incompressible Boussinesq systems: global stabilization under dynamic boundary conditions, \textit{J. Differential Equations} \textbf{367} (2023), 246--289.
 
\bibitem{61}
W. Yang, Q. Jiu, and J. Wu, Global well-posedness for a class of 2D Boussinesq systems with fractional dissipation, \textit{J. Diff. Equ.} \textbf{257} (2014), 4188--4213.


\bibitem{yud}
V.~I. Yudovich, Non-stationary flows of an ideal incompressible fluid, \textit{Ž. Vyčisl. Mat i Mat. Fiz.} \textbf{3} (1963), 1032--1066.

\bibitem{zhao}
K. Zhao, 2D inviscid heat conductive Boussinesq system in a bounded domain, \textit{Michigan
Math. J.} \textbf{59} (2010), 329--352.

\bibitem{a1}
D. Zhou and J. Wu, $H^1$-uniqueness, stability and decay for anisotropic Navier-Stokes and Boussinesq equations, 
\textit{J. Differential Equations} \textbf{453} (2026), part 5, Paper No. 113930, 44 pp.

\end{thebibliography}
\end{document}